\documentclass{amsart}

\usepackage{epsfig}
\usepackage{amscd}
\usepackage{amsmath, amssymb, amsthm}
\usepackage{graphics}
\usepackage{color}
\usepackage{dsfont}
\newtheorem*{main-theorem}{Main Theorem}
\newtheorem{proposition}{Proposition}[section]
\newtheorem{theorem}{Theorem}
\newtheorem{lemma}[proposition]{Lemma}

\theoremstyle{definition}

\newtheorem{remark}[proposition]{Remark}

\numberwithin{equation}{section}

\def\NN{{\mathbb N}}
\def\11{\mathbf{1}}

\def\reals{{\mathbb R}}
\def\cx{{\mathbb C}}

\def\Ci{{\mathcal C}^\infty}
\def\Re{\,\mathrm{Re}\,}
\def\Im{\,\mathrm{Im}\,}

\def\supp{\mathrm{supp}\,}

\def\id{\,\mathrm{id}\,}

\def\O{{\mathcal O}}
\def\SS{{\mathbb S}}
\def\s{{\mathcal S}}

\def\phi{\varphi}
\def\half{{\frac{1}{2}}}
\def\dist{\text{dist}\,}

\def\be{\begin{eqnarray*}}
\def\ee{\end{eqnarray*}}
\def\ben{\begin{eqnarray}}
\def\een{\end{eqnarray}}

\def\lll{\left\langle}
\def\rrr{\right\rangle}

\def\L2R{L_{\text{Rest}}^2}

\def\tchi{\tilde{\chi}}
\def\L2c{L^2_{\text{comp}}}

\def\tA{\tilde{A}}
\def\Dom{\text{Dom}\,}

\begin{document}
\title[Resolvents and the Wave Equation]{Applications of Cutoff
  Resolvent Estimates to the Wave Equation}
\author{Hans Christianson}
\address{Department of Mathematics, MIT, 77 Massachusetts Ave.,
  Cambridge, MA USA}
\email{hans@math.mit.edu}

\begin{abstract}
We consider solutions to the linear wave equation on
non-compact Riemannian manifolds without boundary when the geodesic
flow admits a filamentary hyperbolic trapped set.  We obtain a
polynomial rate of local energy decay with exponent
depending only on the dimension.
\end{abstract}
\maketitle


\section{Introduction}
In this paper we consider solutions to the linear wave equation on the
non-compact Riemannian manifolds with trapping studied by Nonnenmacher-Zworski
\cite{NZ}.  
Let $(X,g)$ be a Riemannian manifold of odd dimension $n\geq 3$ without
boundary, with
(non-negative) 
Laplace-Beltrami operator $- \Delta$ acting on functions.  The Laplace-Beltrami
operator is an unbounded, essentially self-adjoint operator on
$L^2(X)$ with domain $H^2(X)$.  We assume $(X,g)$ is asymptotically
Euclidean in the sense of \cite[(3.7)-(3.9)]{NZ}.  That is there
exists $R_0>0$ sufficiently large that, on
each infinite branch of $M \setminus B(0,R_0)$, the semiclassical
Laplacian $-h^2 \Delta$ takes the form
\be
-h^2 \Delta|_{M \setminus B(0,R_0)} = \sum_{|\alpha| \leq 2}
a_\alpha(x,h) (hD_x)^\alpha,
\ee
with $a_\alpha(x,h)$ independent of $h$ for $|\alpha|=2$, 
\be
&& \sum_{|\alpha| = 2}
 a_\alpha(x,h) (hD_x)^\alpha \geq C^{-1} |\xi|^2, \,\,\, 0 < C <
\infty, \text{ and}\\
&& \sum_{|\alpha| \leq 2}
a_\alpha(x,h) (hD_x)^\alpha \to |\xi|^2, \,\,\, \text{as } |x| \to
\infty \text{ uniformly in } h.
\ee
In order to quote the results of \cite{NZ} we also need the following
analyticity assumption: $\exists \theta_0 \in [0, \pi)$ such that the
$a_\alpha(x,h)$ are extend holomorphically to
\be
\{ r \omega : \omega \in \cx^n, \,\, \dist ( \omega, \SS^n) <
\epsilon, \,\, r \in \cx, \,\, |r| \geq R_0, \,\, \arg r \in
[-\epsilon, \theta_0+ \epsilon) \}.
\ee
As in \cite{NZ}, the analyticity assumption immediately implies 
\be
\partial_x^\beta \left( \sum_{|\alpha| \leq 2} a_\alpha(x, h)
  \xi^\alpha - |\xi|^2 \right)  = o ( |x|^{-|\beta|} )\lll \xi \rrr^2,
\,\, |x| \to \infty.
\ee


We assume also that the classical resolvent $(-\Delta - \lambda^2)^{-1}$ has a holomorphic
continuation to a neighbourhood of $\lambda \in \reals$ as a bounded operator
$L^2_{\text{comp}} \to L^2_{\text{loc}}$.  

We consider solutions $u$ to the following wave equation on $X \times \reals_t$.
\ben
\label{wave-eq-101}
\left\{ \begin{array}{l} (-D_t^2 - \Delta  ) u(x,t) = 0, \,\,\, (x,t) \in X
    \times [0,\infty) \\
u(x,0) = u_0 \in H^1(X) \cap \Ci_c(X), \\
 D_t u(x,0) = u_1 \in
L^2(X) \cap \Ci_c(X), \end{array}
\right.
\een

For $u$ satisfying
\eqref{wave-eq-101} and $\chi \in \Ci_c(X)$, we define the {\it local energy}, $E_{\chi}(t)$, to be
\be
E_{\chi}(t) = \half \left( \left\| \chi \partial_t u
\right\|_{L^2(X)}^2 + \left\| \chi u \right\|_{H^{1}(X)}^2 \right)   .
\ee 
Local energy for solutions to the wave equation has been well studied in various settings.  Morawetz
\cite{Mor}, Morawetz-Phillips \cite{MoPh}, and
Morawetz-Ralston-Strauss \cite{MRS} study the wave equation in
non-trapping exterior domains in $\reals^n$, showing the local energy
decays exponentially in odd dimensions $n \geq 3$, and polynomially in
even dimensions.  This has been
generalized to cases with non-trapping potentials \cite{Vai} and
compact non-trapping perturbations of Euclidean space \cite{Vod}.  In
the case of elliptic trapped rays, it is known that (see \cite{Ral})
exponential decay of the local energy is generally not possible.
Ikawa \cite{I1,I2} shows in dimension $3$ there is exponential local energy decay with
a loss in derivatives in the presence of trapped rays between convex
obstacles, provided the obstacles are sufficiently small and far apart.  In the case $X$ is Euclidean outside a compact set,
$\partial X \neq \emptyset$, and with no assumptions on trapping, Burq shows in \cite{Bur1a} that $E_\chi(t)$ decays at least
logarithmically with some loss in derivatives.  The author
shows in \cite{Ch3} that if there is one
hyperbolic trapped orbit with no other trapping, then the local energy decays exponentially with a
loss in derivative (including the case $\partial X = \emptyset$).  

The main result of this paper is that if there is
a hyperbolic trapped set which is sufficiently ``thin'', then the
local energy decays at least polynomially, with an exponent depending
on the dimension $n$.  

\begin{theorem}
\label{wave-decay-thm}
Suppose $(X,g)$ satisfies the assumptions of the introduction, $\dim X
= n \geq 3$ is odd, and $(X,g)$ admits a compact hyperbolic fractal trapped set, $K_E$, in the energy level
$E>0$ with topological pressure $P_E(1/2)<0$.  Assume there is no
other trapping and $(- \Delta - \lambda^2)^{-1}$ admits a
holomorphic continuation to a strip around $\reals \subset \cx$.  Then
for each $\epsilon>0$ and $s>0$, there is a constant $C >0$, depending
on $\epsilon$, $s$, and the support of $u_0$ and $u_1$, such that
\ben
\label{loc-en-est-101}
E_{\chi}(t) \leq C \left( \frac{ \log (2 + t) }{t}
\right)^{\frac{2s}{3n + \epsilon}} \left( \|u_0\|_{H^{1+s}(X)}^2 +
  \|u_1\|_{H^s(X)}^2 \right).
\een
\end{theorem}

\begin{remark}
It is expected that Theorem \ref{wave-decay-thm} is not optimal, and
in fact an exponential or sub-exponential estimate holds.  Similar to
in \cite{Ch3}, we expect applications to the nonlinear wave equation,
although there are certain technical difficulties to overcome.
\end{remark}

The proof of Theorem \ref{wave-decay-thm} is a consequence of an
adaptation of \cite[Th\'eor\`eme 1]{Bur1a} to this setting and the
following resolvent estimates.

\begin{theorem}
\label{res-lam}
Suppose $(X,g)$ satisfies the assumptions of Theorem
\ref{wave-decay-thm}.  Then for any $\chi \in \Ci_c(X)$ and any $\epsilon>0$ there is a
constant $C = C_{\chi, \epsilon}>0$ such that
\be
\| \chi ( - \Delta - \lambda^2 )^{-1} \chi \|_{L^2(X) \to L^2(X)}
\leq C \frac{ \log( 1 + \lll \lambda \rrr)}{ \lll \lambda \rrr},
\ee
for 
\be
\lambda \in \left\{ \lambda : | \Im \lambda | \leq
  \left\{ \begin{array}{ll} C, & | \Re \lambda | \leq C, \\
      C'| \Re \lambda |^{-3n/2 - \epsilon}, & | \Re \lambda | \geq C 
\end{array} \right. \right\}.
\ee
\end{theorem}

\begin{remark}
The proof of Theorem \ref{wave-decay-thm} depends more on
the neighbourhood in which the resolvent estimates hold than on the 
estimates themselves.  Given a complex neghbourhood of the real axis,
any polynomial cutoff resolvent estimate will give the same local energy 
decay rate.  Theorem \ref{res-lam} represents a gain over
the estimates in \cite[Theorem 5]{NZ} in the sense that the estimate
holds in a complex neighbourhood of $\reals$, rather than just on $\reals$.
\end{remark}

{\bf Acknowledgments.} 
This
research was partially conducted during the period the author was
employed by the Clay Mathematics Institute as a Liftoff Fellow.


\section{Proof of Theorem \ref{res-lam}}
To prove Theorem \ref{res-lam}, we use the results of
Nonnenmacher-Zworski \cite{NZ} to prove a high energy estimate for the
resolvent with complex absorbing potential, then
use the holomorphic continuation to bound the cutoff resolvent by a
constant for low energies.  If we consider the problem
\ben
\label{evp}
(- \Delta - \lambda^2)u  = f,
\een 
and restrict our attention to values $|\lambda | \geq C$ for some
constant $C>0$, we can transform equation \eqref{evp} into a
semiclassical problem for fixed energy by setting 
\be
\lambda = \sqrt{z}/h
\ee
for $z \sim 1$ and $0 < h \leq h_0$.  Then \eqref{evp} becomes
\be
(P -z ) u = h^2 f,
\ee
where
\be
P = -h^2 \Delta
\ee
is the semiclassical Laplacian.

The following Proposition is the
high energy resolvent estimate from \cite{NZ} with the
improvement that the estimate holds in a larger neighbourhood of
$\reals \subset \cx$.

\begin{proposition}
\label{res-nbhd}
Suppose $W \in \Ci( X ; [0,1])$, $W \geq 0$ satisfies
\be
\supp W \subset X \setminus B(0,R_1), \,\,\, W \equiv 1 \text{ on } X
\setminus B(0,R_2),
\ee
for $R_2>R_1$ sufficiently large, and 
\be
\| ( P -iW -z )^{-1} \|_{L^2 \to L^2} \leq C_N \left( 1 + \log (1/h) +
  \frac{h^N}{\Im z} \right),
\ee
for $z \in [E-\delta, E + \delta] + i(-ch, ch)$.  Then for each
$\epsilon>0$ and each $\chi \in
\Ci_c(X)$, there is a constant $C = C_{\epsilon, \chi}>0$ such that 
\be
\|\chi (P-z)^{-1} \chi \|_{L^2 \to L^2} \leq C \frac{ \log (1/h) }{h},
\ee
for $z \in [E-c_1h, E+c_1h] + i(-c_2h^{3n/2 +1 + \epsilon}, c_2
h^{3n/2 + 1 +
  \epsilon} )$.
\end{proposition}

We first improve \cite[Lemma 9.2]{NZ} in order to get cutoff resolvent
estimates with the absorbing potential in a polynomial neighbourhood of the real axis.  The proof
of the following lemma is an adaptation of the ``three-lines'' theorem
from complex analysis and borrows techniques from \cite{Ch, BZ, NZ}
and the references cited therein.

\begin{lemma}
\label{holo-lemma-101}
Suppose $F(z)$ is holomorphic on 
\be
\Omega = [-1,1] + i(-c_-, c_+),
\ee
and satisfies
\be
\log |F(z)| & \leq & M, \,\,\, z \in  \Omega, \\
|F(z)| & \leq & \alpha + \frac{ \gamma}{\Im z}, \,\,\, z \in \Omega
\cap \{ \Im z >0 \}.
\ee
Then if $\gamma \leq \epsilon M^{-3/2}$ for $\epsilon>0$ sufficiently
small, there exists a constant $C = C_\epsilon >0$ such that
\be
|F(z)| \leq C \alpha, \,\,\, z \in [-1/2, 1/2] + i(-M^{-3/2},
M^{-3/2}).
\ee
\end{lemma}

\begin{proof}
Choose $\psi(x) \in \Ci_c([-1,1])$, $\psi \equiv 1$ on
$[-1/2,1/2]$, and set
\be
\phi(z) = \beta^{-1/2} \int e^{-(x-z-ic \beta)^2/\beta} \psi(x) dx,
\ee
where $0 < \beta <1$ and $c>0$ will be chosen later.  The function
$\phi(z)$ enjoys the following properties:

(a) $\phi(z)$ is holomorphic in $\Omega$,

(b) $|\phi(z)| \leq C$ on $\Omega \cap \{ | \Im z| \leq
\beta^{1/2} \}$,

(c) $|\phi(z)| \geq C^{-1}$ on $\{ | \Re z | \leq 1/2 \} \cap \{ | \Im z | \leq \beta
\}$ if $c>0$ is chosen appropriately,

(d) $|\phi(z)| \leq C e^{-C/\beta}$ for $z \in \{ \pm 1 \} +
i(-\beta^{1/2}, \beta^{1/2})$.

Now for $a \in \reals$ to be determined, set 
\be
g(z) = e^{iaz} \phi(z) F(z).
\ee
For $\delta_\pm >0$ to be determined, let 
\be
\Omega' := \Omega \cap \{ - \delta_- \leq \Im z \leq \delta_+\}.
\ee
We have the following bounds for $g(z)$ on the boundary of $\Omega'$:
\be
\log |g(z) | \leq \left\{ \begin{array}{ll} -C/\beta + M - a \Im z, &
    \Re z = \pm 1, \,\, \text{if } |\Im z | \leq \beta^{1/2}, \\
    C + M + a \delta_-, & \Im z = - \delta_- \geq -\beta^{1/2}, \\
    C + \log ( \alpha + \gamma/\delta_+) - a \delta_+, & \Im z =
    \delta_+ \leq \beta^{1/2}. \end{array} \right.
\ee
We want to choose $a$, $\beta$, and $\delta_\pm$ to optimize these
inequalities.  Choosing $a = -2M / \delta_-$ yields
\be
\log |g(z) | \leq C-M \text{ for } \Im z = - \delta_- ,
\ee
and choosing $\delta_+ = | 2/a|$ yields
\be
\log |g(z) | \leq C + \log ( \alpha + \gamma/\delta_+) + 2, \text{ for
} \Im z =
    \delta_+ .
\ee
Finally, chooing $\beta = C'/M$ for an appropriate $C'>0$ yields
\be
\log |g(z) | \leq -C^{-1} M \text{ for } \Re z = \pm 1, \,\, |\Im z |
\leq \max \{ \delta_+, \delta_- \},
\ee
and taking $\delta_- = C'' M^{-1/2}$, $\delta_+ = C'' M^{-3/2}$ gives
\be
\log |g(z)| \leq C''' + \log ( \alpha + \gamma/\delta_+) \text{ on
} \partial \Omega'.
\ee

In order to conclude the stated inequality on $F(z)$, we need to
invert $e^{-iaz} \phi(z)$, which, from the definition of $a$ and the
properties of $\phi$ stated above, is possible for
\be
z \in [-1/2, 1/2] + i(-M^{-3/2}, M^{-3/2}).
\ee
Then for $z$ in this range and $\gamma$ satisfying $\gamma \leq
\epsilon M^{-3/2}$,
\be
|F(z)| \leq C \alpha (1 + \epsilon) \leq C' \alpha,
\ee
as claimed.
\end{proof}

Now to prove Proposition \ref{res-nbhd}, as in \cite{NZ}, we apply Lemma
\ref{holo-lemma-101} to
\be
F(\zeta) = \lll h (P -iW -h \zeta)^{-1} f, g \rrr_{L^2},
\ee
for $f,g \in L^2$.  For $M$ we use the well-known estimate
\be
\| (P -iW -z)^{-1} \|_{L^2 \to L^2} \leq C_\epsilon e^{C
  h^{-n-\epsilon}}, \,\,\, \Im z \geq -h/C,
\ee
and take $M = C_\epsilon h^{-n-\epsilon}$.  For the other parameters,
we take
\be
 \gamma = h^N, \,\,\,\, \alpha = c_0 + \log(1/h).
\ee
Rescaling, we conclude
\be
\| (P - iW -z )^{-1} \| \leq C \frac{ \log(1/h)}{h}
\ee
in the stated region.  Then we apply the remainder of the proof
\cite[Theorem 5]{NZ}.
\qed


\section{Proof of Theorem \ref{wave-decay-thm}}
In this section we adapt the proof of \cite[Th\'eor\`eme 1]{Bur1a} to
the case where one has better resolvent estimates.  We first present a
general theorem on semigroups (see \cite[Th\'eor\`eme 3]{Bur1a} and \cite{Leb}).

Let $H$ be a Hilbert space, $B(\xi)$ a meromorphic family of unbounded
linear operators on $H$, holomorphic for $\Im \xi <0$.  Assume for
$\Im \xi \leq 0$,
\be
\Im (B(\xi) u,u )_H \geq 0.
\ee
Let $\Dom (B) = \Dom (1-iB(-i))$ denote the domain of $B$.  Assume for $\Im
\xi <0$, $\xi - B(\xi)$ is bijective and bounded with respect to the
natural norm on $\Dom (B)$,
\be
\| u \|^2_{\Dom (B)} = \| u \|^2_H + \| B(-i) u \|^2_H, 
\ee
and
\be
\| ( \xi - B(\xi))^{-1} \|_{H \to H} \leq C | \Im \xi |^{-1}.
\ee
Assume that $B(\xi) \in \s^1 ( \reals^2; \mathcal{L}(\Dom (B), H))$.  That
is, $B(\xi)$ is a symbol with respect to $\xi$ and assume that, as
operators on $\Dom (B)$,
\be
B(D_s) e^{is \xi}  & =& e^{is \xi } B(\xi + D_s) \\
& = & e^{is \xi } B( \xi ),
\ee
since members of $\Dom (B)$ do not depend on $s$.
We assume $B$ satisfies the identity
\be
B( D_t ) \psi(t) U(t) & = & \psi(t) B( \psi'/i \psi + D_t ) U(t)
\\
& = & \psi(t) (B(D_t) + A(t)) U(t),
\ee
for $\psi(t) \in \Ci ( \reals)$, and $U \in \Ci( \reals_t; \Dom (B)
)$.  Here, $A(t)$ is a linear operator, bounded on $H$ and has compact
support contained in $\supp \psi'$.

By the Hille-Yosida Theorem, for every $k \in \NN$ and $s \geq 0$, we
can construct the operators
\be
\frac{e^{i s B(D_s)}}{(1 - iB(-i))^{k}},
\ee
where $e^{i s B(D_s)}$ satisfies the evolution equation
\be
\left\{ \begin{array}{l} (D_s - B(D_s)) e^{i s B(D_s)} = 0, \\ e^{i s
      B(D_s)}|_{s=0} = \id. \end{array} \right.
\ee

Now suppose $\chi_j$, $j=1,2$ are bounded operators $H \to H$, and
$\chi_1 (\xi - B(\xi))^{-1} \chi_2$ continues holomorphically to the
region
\be
\Omega = \left\{ \xi \in \cx : | \Im \xi | \leq
  \left\{ \begin{array}{ll} C, & | \Re \xi | \leq C \\
      P(| \Re \xi | ), & | \Re \xi | \geq C, \end{array}
  \right. \right \},
\ee
where $P( | \Re \xi |) >0$ and is monotone decreasing (or constant) as $| \Re \xi |
\to \infty$.  Assume
\ben
\label{res-est-101}
\| \chi_1 (\xi - B(\xi))^{-1} \chi_2 \|_{H \to H} \leq G( | \Re \xi |)
\een
for $\xi \in \Omega$, where $G(| \Re \xi| ) = \O( | \Re \xi |^N)$ for
some $N \geq 0$.  We further assume that the propagator $e^{isB(D_s)}$
``acts finitely locally,'' in the sense that for $s \in [0,1]$,
\be
\tchi_2 := e^{isB(D_s)} \chi_2
\ee
is also a bounded operator on $H$, and $\chi_1 (\xi - B( \xi ))^{-1}
\tchi_2$ continues holomorphically to $\Omega$ and satisfies the
estimate \eqref{res-est-101} with $G$ replaced by $CG$ for a constant $C>0$.

\begin{theorem}
\label{burq-thm}
Suppose $B( \xi )$ satisfies all the assumptions above, and let $k \in \NN$,
$k > N+1$.  Then for any $F(t)>0$, monotone increasing, 
satisfying
\ben
\label{F-cond}
F(t)^{k+1} \leq \exp (t P(F(t))),
\een
we have
\ben
\label{abs-decay-est}
\left\| \chi_1 \frac{ e^{it B(D_t)}}{(1 - iB(-i))^k} \chi_2
\right\|_{H \to H} \leq C F(t)^{-k}.
\een
\end{theorem}

As in \cite{Bur1a}, Theorem \ref{wave-decay-thm} follows from Theorem
\ref{burq-thm} by setting
\be
B = \left( \begin{array}{cc} 0 & -i \id \\ -i \Delta & 0 \end{array}
\right),
\ee
the Hilbert space $H = H^1(X) \times L^2(X)$, and $\chi_j \in \Ci_c(X)$
for $j = 1,2$.  The commutator $[\chi_2, B]$ is compactly
supported and bounded on $H$, so if $\tchi_2\in \Ci_c(X)$ is supported on a slightly larger
set than $\chi_2$, we have
\be
\| \chi_1 e^{itB} \chi_2 \|_{\Dom (B^k) \to  H} & = & \| \chi_1 e^{itB} \chi_2
(1 -iB)^{-k}\|_{H \to H} \\
& \leq & C \| \chi_1 e^{itB} 
(1 -iB)^{-k} \tchi_2 \|_{H \to H}.
\ee
Taking $k = 2$, $P(t) = t^{-3n/2 - \epsilon/2}$, and 
\be
F(t) = \left( \frac{t}{\log t} \right)^{\frac{2}{3n + \epsilon}}
\ee
yields \eqref{loc-en-est-101} for $s \geq k$.  We observe the spaces
$H^{1+s} \times H^s$ are complex interpolation spaces, hence
interpolating 
with the trivial estimate
\be
E_{\chi}(t) \leq \|u_0 \|_{H^1}^2 + \|u_1 \|_{L^2}^2,
\ee
yields \eqref{loc-en-est-101} for $s \geq 0$.

\qed

\begin{remark} 
Evidently, if we have polynomial resolvent bounds in a fixed strip
around the real axis, we have exponential local energy decay for the
wave equation with a loss in derivatives.  Further, if $H = L^2(X)$ for $X$ a compact manifold,
this theorem may be applied to the damped wave operator with $\chi_1 =
\chi_2 = 1$ to conclude there is exponential energy decay with loss in
derivatives for
solutions to the damped wave equation if there is a polynomial bound
on the inverse of the damped wave operator in a strip.  This corrects
a mistake in the proof of \cite[Theorem 5]{Ch}.
\end{remark}

We first need a lemma.
\begin{lemma}
\label{prop-lemma}
For $k > N+1$, the propagator satisfies the following identity on $H$:
\be
\frac{e^{it B(D_t)}}{(1-i B(-i))^k} = \frac{1}{2 \pi i} \int_{ \Im \xi
  = - \half} e^{it \xi} ( 1 - i \xi )^{-k} ( \xi - B( \xi))^{-1} d
\xi.
\ee
\end{lemma}

\begin{proof}
We write $I_k$ for the right hand side and observe both
the left hand side and $I_k$ satisfy the evolution equation
\be
(D_t - B(D_t)) w = 0.
\ee
To calculate $I_k(0)$, we deform the contour to see
\be
I_k(0) = \frac{1}{2 \pi i} \left( \int_{ \Im \xi = -C} -
  \int_{\partial B(-i, \epsilon)} \right) ( 1 - i \xi )^{-k} ( \xi - B( \xi))^{-1} d
\xi.
\ee
Letting $C \to \infty$, the first integral vanishes.  Thus we need to
calculate the second integral.  For $k=1$, this is the residue
formula, while for $k>1$ the formula follows by induction and the
continuity of $B( \xi)$ as $\epsilon \to 0$.

Thus the left hand side and $I_k$ have the same initial conditions,
and the lemma is proved.
\end{proof}

\begin{proof}[Proof of Theorem \ref{burq-thm}]
Now, as in \cite{Bur1a}, we introduce a cutoff in time to make the
equation inhomogeneous, then analyze the integral separately for low
and high frequencies in $\xi$.  In order to maintain smoothness, we
convolve with a Gaussian.  For an initial condition $u_0 \in H$, let
$V(t) = e^{itB(D_t)} \chi_2 u_0$, and consider $U(t) = \psi(t) (1
-iB(-i))^{-k} V(t)$
for $\psi(t) \in \Ci( \reals)$ satisfying $\psi \equiv 0$ for $t \leq
1/3$, $\psi \equiv 1$ for $t \geq 2/3$, and $\psi' \geq 0$.  We
observe by the sub-unitarity of $e^{itB(D_t)}$ for $t \geq 0$,
\be
\| U(t) \| \leq C \| V(t) \| \leq C' \| u_0\|,
\ee
where for the remainder of the proof, $\| \cdot \| = \| \cdot \|_H$
unless otherwise specified.

The family $U(t)$
satisfies
\be
(D_t - B(D_t))U = \tA(t) (1 -i B(-i))^{-k}V(t),
\ee
where $\tA$ is a bounded operator on $H$ with support contained in
$[1/3, 2/3]$.  As $U(0) = 0$, Duhamel's formula yields
\be
U(t) = \int_0^t e^{i(t-s) B(D_t)} \tA(s) (1 -i B(-i))^{-k}V(s) ds,
\ee
and by Lemma \ref{prop-lemma},
\be
U(t) = \int_{s=0}^t \int_{\Im \xi = -1/2} e^{i(t-s) \xi } \tA(s) (1 - i \xi
)^{-k} ( \xi - B( \xi))^{-1} V(s) d \xi d s.
\ee
For a function $F(t)>0$, monotone increasing in $t$ to be selected
later, we will cut off frequencies in $|\xi|$ above and below
$F(t)^2$.  We convolve with a Gaussian to smooth this out:
\be
U(t) & = & \int_{s=0}^t \int_{ \Im \xi = -1/2} \int_\lambda
e^{i(t-s) \xi }\tA(s) (1 - i \xi)^{-k} ( \xi - B( \xi ))^{-1} \\
&& \quad \quad \quad \quad \cdot (c_0 /
\pi)^\half e^{-c_0( \lambda - \xi / F(t) )^2} V(s) d \lambda d \xi ds \\
& = & \int_0^t \int_{ \Im \xi = -1/2} \left( \int_{| \lambda | \leq
    F(t)} + \int_{ |\lambda| \geq F(t) } \right) (\cdot) d \lambda d \xi
ds \\
& =: & I_1 + I_2.
\ee

{\bf Analysis of $I_1$:} From the resolvent and propagator continuation properties, the
integrand in $I_1$ is holomorphic in $\{ \Im \xi <0 \} \cup \Omega$.  Observe
if $| \Re \xi| \gg F(t)^2$, then the integrand is rapidly decaying,
hence we can deform
the contour in $\xi$ to 
\be
\Gamma = \left\{ \xi \in \cx :  \Im \xi  =
  \left\{ \begin{array}{ll} C, & | \Re \xi | \leq C \\
      P(| \Re \xi | ), & | \Re \xi | \geq C. \end{array}
  \right. \right \}
\ee

We further break $I_1$ into integrals where $\Re \xi$ is larger than
or smaller than $F^2(t)$:
\be
I_1 & = & \int_0^t \left( \int_{ \Gamma \cap \{ | \Re \xi | \leq A
    F(t)^2 \}} + \int_{ \Gamma \cap \{ | \Re \xi | \geq A
    F(t)^2 \}} \right)
\int_{| \lambda | \leq
    F(t)} (\cdot) d \lambda d \xi
ds \\ & =:& J_1 + J_2.
\ee

For $J_1$, if $t \geq 2$, since $P( | \Re \xi |)$ is monotone decreasing, we have
\be
\Im \xi \geq P(A F(t)^2),
\ee
and on the support of $\tA$, we have $t-s \geq t-1$.  Hence
\be
\left\| \chi_1 J_1 \right\| & \leq & C  \int_{ \Gamma \cap \{ | \Re \xi | \leq A
    F(t)^2 \}} \int_{| \lambda | \leq
    F(t)} e^{-(t-1) P(A F(t)^2)} \lll \xi \rrr^{-k} G( | \Re \xi | ) \\
&& \quad \quad \quad \quad 
    \cdot \left| e^{-c_0( \lambda - \xi / F(t) )^2} \right| d \lambda d
    \xi \| u_0 \| \\
& \leq & C A F(t)^2 e^{-t P( A F(t)^2)} \| u_0 \|.
\ee

For $J_2$, we observe that for $A$ large enough and $| \Re \xi | \geq
A F(t)^2$, 
\be
\Re ( \lambda - \xi / F(t) )^2 \geq C^{-1} ( \lambda^2 + ( \Re \xi
)^2/ F(t)^2).
\ee
Hence,
\be
\left\| \chi_1 J_2 \right\| & \leq & C  \int_{ \Gamma \cap \{ | \Re \xi | \geq A
    F(t)^2 \}} \int_{| \lambda | \leq
    F(t)} \lll \xi \rrr^{-k} G( | \Re \xi | ) \\ 
&& \quad \quad \quad \quad
    \left| e^{-c_0( \lambda - \xi / F(t) )^2} \right| d \lambda d
    \xi \| u_0 \| \\
& \leq & C \int_{ | \eta | \geq F(t) } F(t) e^{-c_1 \eta^2} d \eta \|
u_0 \| \\
& \leq & C  F(t) e^{-c_2 F(t) } \| u_0 \|.
\ee

{\bf Analysis of $I_2$:} Set
\be
J(\tau) & = & \int_{s=0}^1 \int_{{\Im \xi = -1/2}\atop{|\lambda| \geq F(t)}}
\tA(s) e^{i(\tau-s) \xi} (1 -i \xi )^{-k}( \xi - B( \xi ))^{-1} \\
&& \quad \quad \quad \quad \cdot (c_0 /
\pi)^\half e^{-c_0( \lambda - \xi / F(t) )^2} V(s) d \lambda d \xi
ds,
\ee
which for $\tau \geq 1$ is equal to $U(\tau)$.  Observe 
\be
(D_\tau - B(D_\tau))J( \tau )  &=&  \int_{s=0}^1 \int_{{\Im \xi = -1/2}\atop{|\lambda| \geq F(t)}}
\tA(s) e^{i(\tau-s) \xi} (1 -i \xi )^{-k} \\
&& \quad \quad \quad \quad \cdot (c_0 /
\pi)^\half e^{-c_0( \lambda - \xi / F(t) )^2} V(s) d \lambda d \xi
ds \\
& =:& K(\tau).
\ee
Hence
\be
J( t ) = e^{i t B( D_t )} J(0) + \int_0^t e^{i (t-s) B(
  D_s )} K(s) ds.
\ee
Again, by the subunitarity of the propagator, we need to estimate
$\|J(0) \|$ and $\int_0^t \| K(s) \| ds$.  For $s \in [1,t]$, since $k
>N+1$, we can deform the $\xi$-contour in the definition of $K$ to
$\Im \xi = F(t)$.  Then for this range of $s$,
\be
\|K(s)\|  \leq  C \int_{\eta} e^{-(s-2/3) F(t)} \lll \eta \rrr^{-k} d
\eta \|u_0 \|,
\ee
and hence
\be
\int_1^t \| K(s) \| ds \leq C F(t)^{-1} e^{-F(t)/3}.
\ee

For $J(0)$, we first consider $\lambda \geq F(t)$.  Since $k > N+1$, we can deform the $\xi$-contour to 
\be
\Gamma'  =  \Gamma_- \cup \Gamma_+ 
\ee
where
\be
\Gamma_- & = & \{ \Re \xi \leq F(t)^2/A, \,\, \Im \xi = -1/2 \} \\
&& \quad  \cup \{ \Re \xi =
F(t)^2/A, \,\, -F(t) \leq \Im \leq -1/2 \} 
\ee
and
\be
\Gamma_+ =  \{ \Re \xi \geq
F(t)^2/A, \,\, \Im \xi = -F(t) \} .
\ee
If $\xi \in \Gamma_-$, we have
\be
\Re ( \lambda - \xi / F(t) )^2 \geq  \lambda^2 /C,
\ee
so
\be
\int_{\xi \in \Gamma_-} \int_{\lambda \geq F(t) } \lll \xi \rrr^{-k}
G( | \Re \xi |) \cdot  e^{-c_0( \lambda - \xi / F(t) )^2} V(s) d \lambda d \xi
\leq C e^{-F(t)^2}.
\ee
For $\xi \in \Gamma_+$, we have
\be
| e^{-is \xi }| = e^{-F(t)/3},
\ee
so the contribution to $\|J(0) \|$ coming from $\lambda \geq F(t)$ is
bounded by
\be
C (e^{-F(t)^2} + e^{-F(t)/3} ).
\ee
The contribution to $\| J(0) \|$ coming from $\lambda \leq - F(t)$ is
handled similarly to obtain the same bound.  

We have yet to estimate $\int_0^1 \| K(s) \| ds$.  For this we use
Plancherel's formula to write
\ben
\left( \int_0^1 \| K(s) \|ds \right)^2 & \leq &
\int_{-\infty}^{\infty} \|K(s)\|^2 ds \nonumber \\ & = & \int_{-\infty}^{\infty}
\left\| (1 -i \xi
  )^{-k} \widehat{\tA V}(\xi) \int_{| \lambda | \geq F(t)}
    e^{-c_0(\lambda  - \xi/F(t) )^2} d \lambda \right\|^2 d \xi . \label{eqn-1010}
\een
If we estimate this integral by again considering regions where $| \xi
| \leq F(t)^2/A$ and $| \xi | \geq F(t)^2/A$ respectively, we see \eqref{eqn-1010} is majorized by
\be
\lefteqn{ C( F(t)^{-2k} + e^{-F(t)^2/C}) \int_{-\infty}^{\infty} \left\|
  \widehat{ \tA V} ( \xi ) \right\|^2 d \xi } \\
 & = &
C( F(t)^{-2k} + e^{-F(t)^2/C}) \int_{-\infty}^\infty \|\tA V(s) \|^2 d s \\
& \leq & C( F(t)^{-2k} + e^{-F(t)^2/C}) \|u_0 \|^2.
\ee

Combining all of the above estimates, we have
\be
\|U(t) \| \leq C \max \left\{ \begin{array}{l} F(t)^{-2k} \\
    e^{-F(t)/3} + e^{-F(t)^2/C}, \\
    F(t)^{-1} e^{-F(t)/3}, \\
    F(t)^2 e^{-tP(F(t)^2)} + F(t) e^{F(t)} \end{array} \right\} \|
u_0\|.
\ee
Relabelling $F(t)^2$ as $F(t)$ throughout and applying the condition
\eqref{F-cond}, we recover \eqref{abs-decay-est}.

\end{proof}


\end{document}